\documentclass[12pt,a4paper]{article}
\usepackage[latin1]{inputenc}
\usepackage{amsmath}
\usepackage{amsfonts}
\usepackage{amssymb}
\usepackage{amsthm}

\theoremstyle{plain}
\newtheorem{thm}{Theorem}[section]
\newtheorem{cor}{Corollary}[section]
\newtheorem{prop}{Proposition}[section]

\theoremstyle{definition}
\newtheorem{rmk}{Remark}[section]
\newtheorem{definition}{Definition}[section]

\numberwithin{equation}{section}

\def \I{\hbox{I}}
\def \d{\hbox{d}}

\newcommand{\R}{\mathbb{R}}

\newcommand{\Z}{\mathbb{Z}}

\title{\textbf{Heat Kernel Analysis for Ornstein-Uhlenbeck Operators with Quadratic Potentials }}
\author{Sheng-Ya Feng}
\date{}

\begin{document}
\maketitle

\begin{abstract}
\noindent
In this paper, we study Ornstein-Uhlenbeck operators with quadratic potentials. We use Hamiltonian formalism to characterise the singularities produced by the potentials by finding explicit geodesics of the operators, and obtain the heat kernels via a probabilistic ansatz. All the formulae are closed.

\bigskip
\hspace{-15pt}\textbf{Key Words:} Hamiltonian formalism, Ornstein-Uhlenbeck operator, heat kernel
\end{abstract}

\bigskip
\section{Introduction}

\bigskip
\subsection{Background}

\bigskip
Ornstein-Uhlenbeck operator is a Linear operator 
\begin{equation*}
T=-\Delta+x\cdot\nabla
\end{equation*}
acting on $L^{2}(\R^{n}; \mu_{n})$, where 
\begin{equation*}
\mu_{n}(E)=\frac{1}{(2\pi)^{n/2}}\int_{E}\exp\left( -\frac{|x|^{2}}{2}\right)\d x
\end{equation*}
is the Gaussian measure, $\Delta=\sum_{j=1}^{n}\partial_{x_{j}}^{2}$ and $\nabla=(\partial_{x_{1}}, \partial_{x_{2}}, \ldots, \partial_{x_{n}})$ denote Laplace operator and gradient, and  $x\cdot y=\sum_{j=1}^{n}x_{j}y_{j}$ is the canonical inner product in $\R^{n}$ with an extension $x\cdot\nabla=\sum_{j=1}^{n}x_{j}\partial_{x_{j}}$. 

\bigskip
For $s\geq 0$ and $f\in L^{2}(\R^{n}; \mu_{n})$,
\begin{equation*}
e^{-sT}f(x)=\int_{\R^{n}}f\left(e^{-s}x+\sqrt{1-e^{-2s}}\,y \right)\d \mu_{n}(y) 
\end{equation*}
is the Mehler formulae for the associated Ornstein-Uhlenbeck semigroup $\left\lbrace e^{-sT}\right\rbrace _{s\geq 0}$. Readers may consult \cite{B98}, \cite{FP94} and \cite{H05} for more elementary properties on Ornstein-Uhlenbeck operator and semigroup, and \cite{H99}, \cite{P73}, \cite{N73} for further interests. 

\bigskip
Ornstein-Uhlenbeck operators in the whole space $L^{p}(\R^{n})$ are well understood (cf. \cite{M01}, \cite{MPV05} and \cite{MPRS02}), and for operators on domains, bounded or unbounded, we refer \cite{DPL04}, \cite{FMP04} and \cite{GHHW05} for further reading. Ornstein-Uhlenbeck semigroups also have rich applications to inequalities in geometry and analysis, such as Brascamp-Lieb inequalities (cf. \cite{B89}, \cite{B91} and \cite{BC-E04}), Gaussian logarithmic Sobolev inequality (cf. \cite{L03}) and the reverse Bobkov isoperimetric inequality (cf. \cite{BC-EF01} and \cite{B97}).

\bigskip
\subsection{Motivation, methodology and arrangements}
In this paper, we shall consider Ornstein-Uhlenbeck operators perturbed by quadratic potentials. Essentially, they take the form 
\begin{equation}
L=-\theta \partial_{x}^{2}+ ax\partial_{x}+\rho x^{2}.\label{L}
\end{equation}
with all the parameters real scalars. The study of Hermite type operators (cf. \cite{CF11} and \cite{F12}) show us that negative potentials give rise to singularities. We are naturally concerned with the following questions.
\begin{itemize}
\item How do singularities arise? 
\item How to characterize the singularities?
\end{itemize}

Similar to Ornstein-Uhlenbeck operator, Kolmogorov operator in \cite{K34} was first proposed to describe the probability density of a free system, and its heat kernel was obtained via Fourier transform. H\"{o}rmander generalised the method to an appreciable extent in the introduction of his celebrated paper \cite{H67}, where heat kernels of Ornstein-Uhlenbeck operators were benefited. However, operator in (\ref{L}) is invariant under Fourier transform, which make those methods inapplicable to find its heat kernel. Here, we adopt Hamiltonian formalism \cite{BGG97} and an ansatz \cite{B99} to reach our goals.

In section \ref{geo}, we use Hamiltonian formalism to find the singularities and explicit geodesics of the operators (\ref{L}). Heat kernels are obtained, with the help of ansatz, in section \ref{kernel}. We stress the critical case for singularities to arise and make more discussions on heat kernel in section \ref{discussion}. 

As typical examples, perturbed Ornstein-Uhlenbeck operators
\begin{equation}
L^{+}=-\partial_{x}^{2}+x\partial_{x}+x^{2}\label{L+}
\end{equation}
and 
\begin{equation}
L^{-}=-\partial_{x}^{2}+x\partial_{x}-x^{2} \label{L-}
\end{equation}
will help readers digest our presentation.

\bigskip
\section{Hamiltonian formalism}\label{geo}

\medskip
\subsection{Hamiltonian systems}
For more generalities, we consider 1-d operators
\begin{equation}
L=-\theta \partial_{x}^{2}+(ax+b)\partial_{x}+\rho x^{2},
\end{equation}
where $\theta >0$, $a, b\in \R$, $\rho \in \R \setminus \{0\}$.
The Hamiltonian of the operator is defined as its full symbol
\begin{equation*}
H=-\theta \xi^{2}+(ax+b)\xi+\rho x^{2},
\end{equation*}
and the associated Hamiltonian system is 
\begin{equation}
\left\{\begin{aligned}
& \dot{x}=\frac{\partial H}{\partial \xi}= ax-2\theta \xi+b\\
& \dot{\xi}= -\frac{\partial H}{\partial x}= -2\rho x-a\xi
\end{aligned}\right..\label{hamilton}
\end{equation}
Equivalently,
\begin{equation*}
\begin{bmatrix}
\dot{x}\\
\dot{\xi}
\end{bmatrix}
=
\begin{bmatrix}
a & -2\theta\\
-2\rho & -a
\end{bmatrix}
\begin{bmatrix}
x\\
\xi
\end{bmatrix}
+
\begin{bmatrix}
b\\
0
\end{bmatrix}.
\end{equation*}
Denoting
\begin{equation*}
X=
\begin{bmatrix}
x\\
\xi
\end{bmatrix},
\end{equation*}
we have
\begin{equation}
\dot{X}=AX+B\label{nonhom}
\end{equation}
with
\begin{align*}
A=
\begin{bmatrix}
a & -2\theta\\
-2\rho & -a
\end{bmatrix},\hspace{20pt}
B=
\begin{bmatrix}
b\\
0
\end{bmatrix}.
\end{align*}

\bigskip
\subsection{Homogeneous equation $\dot{X}=AX$ }
In order to solve the Hamiltonian system, we first consider the fundamental matrix of the system, i.e. solutions of the homogeneous equation 
\begin{equation}
\dot{X}=AX.\label{hom}
\end{equation}
By the characteristic equation of $A$
\begin{equation*}
\left |\lambda \I-A \right|=
\left| \begin{array}{cc}
\lambda-a & 2\theta\\ 
2\rho & \lambda+a
\end{array}\right|=
\lambda^{2}-a^{2}-4\rho\theta,
\end{equation*}
we have two eigenvalues, $\lambda_{1}=\sqrt{a^{2}+4\rho\theta}$ and $\lambda_{2}=-\sqrt{a^{2}+4\rho\theta}$. According to the signs of the quantity $a^{2}+4\rho\theta$, namely positive, zero and negative, we find the fundamental matrix case by case.

\bigskip
\subsubsection*{Case 1 \hspace{15pt}$a^{2}+4\rho\theta>0$.  ($\lambda_{1}>0>\lambda_{2}$)}
Fundamental matrix $\Phi(s)$ has the form $\left[ e^{\lambda_{1}s}\eta_{1},  e^{\lambda_{2}s}\eta_{2}\right]$, where $\eta_{i}$ is an eigenvector corresponding to the eigenvalue $\lambda_{i}$. To find $\eta_{i}$'s, we consider the characteristic equations
\begin{align*}
&\qquad \left(\lambda_{i}\I-A \right)y = 0\\
\\
&\Longleftrightarrow 
\begin{bmatrix}
\lambda_{i}-a & 2\theta\\
2\rho & \lambda_{i}+a
\end{bmatrix}
\begin{bmatrix}
y_{1}\\
y_{2}
\end{bmatrix}
=
\begin{bmatrix}
0\\
0
\end{bmatrix}\\
\\
&\Longleftrightarrow 
\left\{\begin{aligned}
& y_{1}=-\frac{2\theta}{\lambda_{i}-a}y_{2}\\
& y_{2}\in \R\setminus\{0\}
\end{aligned}\right..
\end{align*}
Taking $\eta_{1}=\left[ -\frac{2\theta}{\lambda_{1}-a}, 1\right] ^{t}$, $\eta_{2}=\left[ -\frac{2\theta}{\lambda_{2}-a}, 1\right] ^{t}$, we find a fundamental matrix 
\begin{equation}
\Phi(s)=
\begin{bmatrix}
-\frac{2\theta}{\lambda_{1}-a}e^{\lambda_{1}s} & -\frac{2\theta}{\lambda_{2}-a}e^{\lambda_{2}s}\\
 \\
e^{\lambda_{1}s} & e^{\lambda_{2}s}
\end{bmatrix}.\label{fm1}
\end{equation}
The inverse of the fundamental matrix
\begin{align}
\Phi^{-1}(s)
&=\frac{(\lambda_{1}-a)(\lambda_{2}-a)}{2\theta(\lambda_{1}-\lambda_{2})}
\begin{bmatrix}
e^{\lambda_{2}s} & \frac{2\theta}{\lambda_{2}-a}e^{\lambda_{2}s}\\
 & \\
-e^{\lambda_{1}s} & -\frac{2\theta}{\lambda_{1}-a}e^{\lambda_{1}s}
\end{bmatrix}\\
\notag\\
&=-\frac{\rho}{\lambda_{1}}
\begin{bmatrix}
e^{\lambda_{2}s} & \frac{2\theta}{\lambda_{2}-a}e^{\lambda_{2}s}\\
 & \\
-e^{\lambda_{1}s} & -\frac{2\theta}{\lambda_{1}-a}e^{\lambda_{1}s}
\end{bmatrix}\label{fm-1}
\end{align}
will be used later. 

\bigskip
\subsubsection*{Case 2 \hspace{15pt}$a^{2}+4\rho\theta=0$.  ($\lambda_{1}=\lambda_{2}=0$)}
Noting $A$ is nilpotent of step 2, we may take 
\begin{equation}
\Phi(s)=\exp(sA)=\I + sA=
\begin{bmatrix}
1+as & -2\theta s\\
-2\rho s & 1-as
\end{bmatrix}.\label{fm2}
\end{equation}
Similarly, 
\begin{equation}
\Phi^{-1}(s)=\exp(-sA)=
\begin{bmatrix}
1-as & 2\theta s\\
2\rho s & 1+as
\end{bmatrix}.\label{fm-2}
\end{equation}

\bigskip
\subsubsection*{Case 3 \hspace{15pt}$a^{2}+4\rho\theta<0$.  ($\lambda_{1}, \lambda_{2}\in \imath \R, \hspace{5pt}\imath=\sqrt{-1} $)}
In this case, $\lambda_{1}$ and $\lambda_{2}$ are pure imaginaries. Denoting $\lambda_{0}=\sqrt{-a^{2}-4\rho\theta}$, we write $\lambda_{1}=\imath \lambda_{0} $ and $\lambda_{1}=-\imath \lambda_{0}$. As same as case 1, we have fundamental matrix 
\begin{equation*}
\Psi(s)=
\begin{bmatrix}
-\frac{2\theta}{\lambda_{1}-a}e^{\lambda_{1}s} & -\frac{2\theta}{\lambda_{2}-a}e^{\lambda_{2}s}\\
 & \\
e^{\lambda_{1}s} & e^{\lambda_{2}s}
\end{bmatrix},
\end{equation*}
with 
\begin{equation*}
\Psi(0)=
\begin{bmatrix}
-\frac{2\theta}{\lambda_{1}-a} & -\frac{2\theta}{\lambda_{2}-a}\\
 & \\
1 & 1
\end{bmatrix}
\end{equation*}
and
\begin{equation*}
\Psi^{-1}(0)=\frac{(\lambda_{1}-a)(\lambda_{2}-a)}{2\theta(\lambda_{1}-\lambda_{2})}
\begin{bmatrix}
1 & \frac{2\theta}{\lambda_{2}-a}\\
 & \\
-1 & -\frac{2\theta}{\lambda_{1}-a}
\end{bmatrix}.
\end{equation*}
Given a real system (\ref{hom}), we expect to find a real fundamental matrix via the complex matrix $\Psi(s)$. We take
\begin{align}
\notag \Phi(s)&=\exp(sA)\\
\notag\\
\notag &=\Psi(s)\Psi^{-1}(0)\\
\notag\\
\notag &=\frac{(\lambda_{1}-a)(\lambda_{2}-a)}{2\theta(\lambda_{1}-\lambda_{2})}
\begin{bmatrix}
-\frac{2\theta}{\lambda_{1}-a}e^{\lambda_{1}s} & -\frac{2\theta}{\lambda_{2}-a}e^{\lambda_{2}s}\\
 & \\
e^{\lambda_{1}s} & e^{\lambda_{2}s}
\end{bmatrix}
\begin{bmatrix}
1 & \frac{2\theta}{\lambda_{2}-a}\\
 & \\
-1 & -\frac{2\theta}{\lambda_{1}-a}
\end{bmatrix}\\
\notag\\
&=
\begin{bmatrix}
\cos(\lambda_{0}s)+\frac{a}{\lambda_{0}}\sin(\lambda_{0}s) & -\frac{2\theta}{\lambda_{0}}\sin(\lambda_{0}s)\\
 & \\
-\frac{2\rho}{\lambda_{0}}\sin(\lambda_{0}s) & \cos(\lambda_{0}s)-\frac{a}{\lambda_{0}}\sin(\lambda_{0}s)
\end{bmatrix}.\label{fm3}
\end{align}
A direct computation gives 
\begin{align}
\notag\Phi^{-1}(s)&=\exp(-sA)\\
\notag\\
&=
\begin{bmatrix}
\cos(\lambda_{0}s)-\frac{a}{\lambda_{0}}\sin(\lambda_{0}s) & \frac{2\theta}{\lambda_{0}}\sin(\lambda_{0}s)\\
 & \\
\frac{2\rho}{\lambda_{0}}\sin(\lambda_{0}s) & \cos(\lambda_{0}s)+\frac{a}{\lambda_{0}}\sin(\lambda_{0}s)
\end{bmatrix}.\label{fm-3}
\end{align}

\bigskip
\subsection{Geodesics}
Geodesics of the Hamiltonian system (\ref{hamilton}) is the solutions of 
\begin{equation}
\dot{X}=AX+B\tag{\ref{nonhom}}
\end{equation}
with the boundary condition
\begin{align}
x(0)=x_{0},\hspace{20pt}
x(1)=x.\label{bc}
\end{align}
The solutions of the equation (\ref{nonhom}) has the form 
\begin{equation}
X(s)=\Phi(s)C(s)\label{X}
\end{equation}
where $\Phi(s)$ is a fundamental matrix of the homogeneous equation
\begin{equation}
\dot{X}=AX.\tag{\ref{hom}}
\end{equation}
and $C(s)$ is a column vector satisfying the matrix equation 
\begin{equation*}
\Phi(s)\dot{C}(s)=B,
\end{equation*}
i.e.
\begin{equation}
\dot{C}(s)=\Phi^{-1}(s)B.\label{C}
\end{equation}
We denote $C(s)=\left[ C_{1}(s), C_{2}(s)\right] ^{t}$, and $c=\left[ c_{1}, c_{2}\right] ^{t}$ is a constant vector. We next identify the geodesics case by case. 

\bigskip
\subsubsection*{Case 1 \hspace{15pt}$a^{2}+4\rho\theta>0$.  ($\lambda_{1}>0>\lambda_{2}$)}
By (\ref{fm-1}) and (\ref{C}), we have
\begin{align*}
\begin{bmatrix}
\dot{C_{1}}\\
\dot{C_{2}}
\end{bmatrix}
&=\Phi^{-1}(s)B\\
\\
&=-\frac{\rho}{\lambda_{1}}
\begin{bmatrix}
e^{\lambda_{2}s} & \frac{2\theta}{\lambda_{2}-a}e^{\lambda_{2}s}\\
 & \\
-e^{\lambda_{1}s} & -\frac{2\theta}{\lambda_{1}-a}e^{\lambda_{1}s}
\end{bmatrix}
\begin{bmatrix}
b\\
0
\end{bmatrix}\\
\\
&=-\frac{\rho b}{\lambda_{1}}
\begin{bmatrix}
e^{\lambda_{2}s} \\
-e^{\lambda_{1}s}
\end{bmatrix}.
\end{align*}
Hence, 
\begin{equation*}
\begin{bmatrix}
C_{1}(s)\\
C_{2}(s)
\end{bmatrix}
=
\begin{bmatrix}
-\frac{\rho b}{\lambda_{1}\lambda_{2}}e^{\lambda_{2}s}+c_{1}\\
 \\
\frac{\rho b}{\lambda_{1}^{2}}e^{\lambda_{1}s}+c_{2}
\end{bmatrix}.
\end{equation*}
Making use of (\ref{X}) and (\ref{fm1}), one has
\begin{align*}
\begin{bmatrix}
x(s)\\
\xi(s)
\end{bmatrix}
&=\Phi(s)C(s)\\
\\
&=
\begin{bmatrix}
-\frac{2\theta}{\lambda_{1}-a}e^{\lambda_{1}s} & -\frac{2\theta}{\lambda_{2}-a}e^{\lambda_{2}s}\\
 & \\
e^{\lambda_{1}s} & e^{\lambda_{2}s}
\end{bmatrix}
\begin{bmatrix}
-\frac{\rho b}{\lambda_{1}\lambda_{2}}e^{\lambda_{2}s}+c_{1}\\
 \\
\frac{\rho b}{\lambda_{1}^{2}}e^{\lambda_{1}s}+c_{2}
\end{bmatrix}\\
\\
&=
\begin{bmatrix}
-\frac{2\theta}{\lambda_{1}-a}e^{\lambda_{1}s}\left(-\frac{\rho b}{\lambda_{1}\lambda_{2}}e^{\lambda_{2}s}+c_{1} \right)-\frac{2\theta}{\lambda_{2}-a}e^{\lambda_{2}s}\left(\frac{\rho b}{\lambda_{1}^{2}}e^{\lambda_{1}s}+c_{2} \right) \\
 \\
e^{\lambda_{1}s}\left(-\frac{\rho b}{\lambda_{1}\lambda_{2}}e^{\lambda_{2}s}+c_{1} \right)+e^{\lambda_{2}s}\left(\frac{\rho b}{\lambda_{1}^{2}}e^{\lambda_{1}s}+c_{2} \right)
\end{bmatrix}.
\end{align*}
Thus, 
\begin{align}
\notag 
x(s)&=\left[ \frac{2b\rho\theta}{\lambda_{1}\lambda_{2}\left( \lambda_{1}-a\right) }-\frac{2b\rho\theta}{\lambda_{1}^{2}\left( \lambda_{2}-a\right) }\right] e^{\left( \lambda_{1}+\lambda_{2}\right)s }-\frac{2c_{1}\theta}{\lambda_{1}-a}e^{\lambda_{1}s}-\frac{2c_{2}\theta}{\lambda_{2}-a}e^{\lambda_{2}s}\\
\notag\\
&=-\frac{2ab\rho\theta\left( \lambda_{1}-\lambda_{2}\right) }{\lambda_{1}^{2}\lambda_{2}\left( \lambda_{1}-a\right)\left( \lambda_{2}-a\right)}+d_{1}e^{\lambda_{1}s}+d_{2}e^{\lambda_{2}s} \label{x1}
\end{align}
where $d_{1}=-\frac{2c_{1}\theta}{\lambda_{1}-a}$ and $d_{2}=-\frac{2c_{2}\theta}{\lambda_{2}-a}$. 

\bigskip
Putting $\lambda_{0}=\sqrt{a^{2}+4\rho\theta}$, we have
\begin{align*}
\lambda_{1}=\lambda_{0}, \hspace{15pt} \lambda_{2}=-\lambda_{0}, \hspace{15pt} \frac{2ab\rho\theta\left( \lambda_{1}-\lambda_{2}\right) }{\lambda_{1}^{2}\lambda_{2}\left( \lambda_{1}-a\right)\left( \lambda_{2}-a\right)}=\frac{ab}{\lambda_{0}^{2}}.
\end{align*}
The boundary conditions (\ref{bc}) force $d_{i}$'s to satisfy the linear equations
\begin{equation*}
\begin{bmatrix}
1 & 1\\
e^{\lambda_{1}} & e^{\lambda_{2}}
\end{bmatrix}
\begin{bmatrix}
d_{1}\\
d_{2}
\end{bmatrix}
=
\begin{bmatrix}
x_{0}+\frac{ab}{\lambda_{0}^{2}}\\
 \\
x+\frac{ab}{\lambda_{0}^{2}}
\end{bmatrix}.
\end{equation*}
One substitutes the solution
\begin{equation*}
\begin{bmatrix}
d_{1}\\
d_{2}
\end{bmatrix}
=\frac{1}{e^{-\lambda_{0}}-e^{\lambda_{0}}}
\begin{bmatrix}
e^{-\lambda_{0}}\left( x_{0}+\frac{ab}{\lambda_{0}^{2}}\right)-\left( x+\frac{ab}{\lambda_{0}^{2}}\right)\\
 \\
-e^{\lambda_{0}}\left( x_{0}+\frac{ab}{\lambda_{0}^{2}}\right)+\left( x+\frac{ab}{\lambda_{0}^{2}}\right) 
\end{bmatrix}
\end{equation*}
into (\ref{x1}) and finally obtains the explicit geodesics
\begin{equation}
x(s) =-\frac{ab}{\lambda_{0}^{2}}+\left( x_{0}+\frac{ab}{\lambda_{0}^{2}}\right)\frac{\sinh \left( \lambda_{0}(1-s)\right) }{\sinh \lambda_{0}} +\left( x+\frac{ab}{\lambda_{0}^{2}}\right)\frac{\sinh \left( \lambda_{0}s\right) }{\sinh \lambda_{0}}, \hspace{5pt} 0 \leq s \leq 1.\label{g1}
\end{equation}

\bigskip
\begin{rmk}\label{rmk1}
It is apparent that $\rho >0$ ensures $a^{2}+4\rho\theta >0$, and hence there exists unique geodesics connecting two given points. This fact coincides well with Hermitian case, where positive potentials keep the geodesics regular.  However, $\rho <0$ do not necessarily  produce singularities in our case, as long as $a^{2}+4\rho\theta>0$. 
\end{rmk}

\bigskip
\subsubsection*{Case 2 \hspace{15pt}$a^{2}+4\rho\theta=0$.  ($\lambda_{1}=\lambda_{2}=0$)}
By (\ref{fm-2}) and (\ref{C}), we have 
\begin{align*}
&\qquad \dot{C}=\Phi^{-1}(s)B\\
\\
&\Longleftrightarrow
\begin{bmatrix}
\dot{C_{1}}\\
\dot{C_{2}}
\end{bmatrix}
=
\begin{bmatrix}
1-as & 2\theta s\\
2\rho s & 1+as
\end{bmatrix}
\begin{bmatrix}
b\\
0
\end{bmatrix}\\
\\
&\Longleftrightarrow
\begin{bmatrix}
C_{1}(s)\\
C_{2}(s)
\end{bmatrix}
=
\begin{bmatrix}
bs-\frac{ab}{2}s^{2}+c_{1}\\
 \\
\rho b s^{2}+c_{2}
\end{bmatrix}.
\end{align*}
Hence, by (\ref{X}) and (\ref{fm2}), one has
\begin{align*}
\begin{bmatrix}
x(s)\\
\xi(s)
\end{bmatrix}
&=\Phi(s)C(s)\\
\\
&=
\begin{bmatrix}
1+as & -2\theta s\\
-2\rho s & 1-as
\end{bmatrix}
\begin{bmatrix}
bs-\frac{ab}{2}s^{2}+c_{1}\\
 \\
\rho b s^{2}+c_{2}
\end{bmatrix}\\
\\
&=
\begin{bmatrix}
\left(1+as \right) \left( bs-\frac{ab}{2}s^{2}+c_{1}\right)-2\theta s\left( \rho b s^{2}+c_{2}\right) \\
 \\
-2\rho s \left( bs-\frac{ab}{2}s^{2}+c_{1}\right)+\left( 1-as\right) \left( \rho b s^{2}+c_{2}\right) 
\end{bmatrix}.
\end{align*}
Noting that $a^{2}+4\rho\theta=0$, one gets
\begin{equation}
x(s)=\frac{ab}{2}s^{2}+d_{2}s+d_{1}, \label{x2}
\end{equation}
where $d_{1}=c_{1}$ and $d_{2}=b+ac_{1}-2\theta c_{2}$ satisfy the linear equations implied by the boundary conditions (\ref{bc})
\begin{align*}
\left\{\begin{aligned}
& d_{1}=x_{0}\\
& d_{1}+d_{2}=x-\frac{ab}{2}
\end{aligned}\right.
\hspace{15pt} \Longrightarrow \hspace{15pt}
\left\{\begin{aligned}
& d_{1}=x_{0}\\
& d_{2}=x-x_{0}-\frac{ab}{2}
\end{aligned}\right..
\end{align*}
After substituting $d_{i}$'s into (\ref{x2}), we finally obtain the geodesics 
\begin{equation}
x(s)=\frac{ab}{2}s^{2}+\left( x-x_{0}-\frac{ab}{2}\right) s+x_{0}, \hspace{15pt} 0 \leq s \leq 1. \label{g2}
\end{equation}

\bigskip
\subsubsection*{Case 3 \hspace{15pt}$a^{2}+4\rho\theta<0$.  ($\lambda_{1}, \lambda_{2}\in \imath \R, \hspace{5pt}\imath=\sqrt{-1} $)}
By (\ref{fm-3}) and (\ref{C}), we have 
\begin{align*}
&\qquad \dot{C}=\Phi^{-1}(s)B\\
\\
&\Longleftrightarrow
\begin{bmatrix}
\dot{C_{1}}\\
\dot{C_{2}}
\end{bmatrix}
=
\begin{bmatrix}
\cos(\lambda_{0}s)-\frac{a}{\lambda_{0}}\sin(\lambda_{0}s) & \frac{2\theta}{\lambda_{0}}\sin(\lambda_{0}s)\\
 & \\
\frac{2\rho}{\lambda_{0}}\sin(\lambda_{0}s) & \cos(\lambda_{0}s)+\frac{a}{\lambda_{0}}\sin(\lambda_{0}s)
\end{bmatrix}
\begin{bmatrix}
b\\
0
\end{bmatrix}\\
\\
&\Longleftrightarrow
\begin{bmatrix}
C_{1}(s)\\
C_{2}(s)
\end{bmatrix}
=
\begin{bmatrix}
\frac{b}{\lambda_{0}}\sin\left(\lambda_{0}s \right) +\frac{ab}{\lambda_{0}^{2}}\cos\left(\lambda_{0}s \right)+c_{1}\\
 \\
-\frac{2\rho b}{\lambda_{0}^{2}}\cos\left(\lambda_{0}s \right)+c_{2}
\end{bmatrix}.
\end{align*}
Hence, by (\ref{X}) and (\ref{fm3}), one has
\begin{align*}
\begin{bmatrix}
x(s)\\
\xi(s)
\end{bmatrix}
&=\Phi(s)C(s)\\
\\
&=
\begin{bmatrix}
\cos(\lambda_{0}s)+\frac{a}{\lambda_{0}}\sin(\lambda_{0}s) & -\frac{2\theta}{\lambda_{0}}\sin(\lambda_{0}s)\\
 & \\
-\frac{2\rho}{\lambda_{0}}\sin(\lambda_{0}s) & \cos(\lambda_{0}s)-\frac{a}{\lambda_{0}}\sin(\lambda_{0}s)
\end{bmatrix}\\
&\quad \cdot
\begin{bmatrix}
\frac{b}{\lambda_{0}}\sin\left(\lambda_{0}s \right) +\frac{ab}{\lambda_{0}^{2}}\cos\left(\lambda_{0}s \right)+c_{1}\\
 \\
-\frac{2\rho b}{\lambda_{0}^{2}}\cos\left(\lambda_{0}s \right)+c_{2}
\end{bmatrix},
\end{align*}
and an arrangement gives 
\begin{align}
\notag x(s)&= \left(\cos(\lambda_{0}s)+\frac{a}{\lambda_{0}}\sin(\lambda_{0}s)\right) \left( \frac{b}{\lambda_{0}}\sin\left(\lambda_{0}s \right) +\frac{ab}{\lambda_{0}^{2}}\cos\left(\lambda_{0}s \right)+c_{1}\right)\\ 
\notag &\quad
-\frac{2\theta}{\lambda_{0}}\sin(\lambda_{0}s)\left(-\frac{2\rho b}{\lambda_{0}^{2}}\cos\left(\lambda_{0}s \right)+c_{2} \right) \\
\notag\\
&=\frac{ab}{\lambda_{0}^{2}}+\left(\cos(\lambda_{0}s)+\frac{a}{\lambda_{0}}\sin(\lambda_{0}s)\right)c_{1}-\frac{2\theta \sin\left( \lambda_{0}s\right) }{\lambda_{0}}c_{2} \label{x3}
\end{align}
with $c_{i}$'s fulfilling the boundary conditions (\ref{bc})
\begin{equation}
\begin{bmatrix}
1 & 0\\
 \\
\cos\lambda_{0}+\frac{a}{\lambda_{0}}\sin\lambda_{0} & -\frac{2\theta \sin\lambda_{0}}{\lambda_{0}}
\end{bmatrix}
\begin{bmatrix}
c_{1}\\
c_{2}
\end{bmatrix}
=
\begin{bmatrix}
x_{0}-\frac{ab}{\lambda_{0}^{2}}\\
 \\
x-\frac{ab}{\lambda_{0}^{2}}
\end{bmatrix}.\label{singular Coefficient}
\end{equation}

\bigskip
$\bullet$ \hspace{5pt} $\sin\lambda_{0}\neq 0$, $\left\lbrace c_{i} \right\rbrace_{i=1}^{2} $ have a unique solution.

\bigskip
Indeed, 
\begin{align*}
\begin{bmatrix}
c_{1}\\
c_{2}
\end{bmatrix}
&=
\begin{bmatrix}
1 & 0\\
 \\
\cos\lambda_{0}+\frac{a}{\lambda_{0}}\sin\lambda_{0} & -\frac{2\theta \sin\lambda_{0}}{\lambda_{0}}
\end{bmatrix}^{-1}
\begin{bmatrix}
x_{0}-\frac{ab}{\lambda_{0}^{2}}\\
 \\
x-\frac{ab}{\lambda_{0}^{2}}
\end{bmatrix}\\
 \\
&=
\begin{bmatrix}
x_{0}-\frac{ab}{\lambda_{0}^{2}}\\
 \\
\frac{\lambda_{0}}{2\theta \sin\lambda_{0}}\left( \cos\lambda_{0}+\frac{a}{\lambda_{0}}\sin\lambda_{0}\right)\left(x_{0}-\frac{ab}{\lambda_{0}^{2}} \right)-\frac{\lambda_{0}}{2\theta \sin\lambda_{0}}\left(x-\frac{ab}{\lambda_{0}^{2}} \right) 
\end{bmatrix}.
\end{align*}
Substituting $c_{i}$'s back to (\ref{x3}) and making an arrangement, one formulates the geodesics
\begin{equation}
x(s) =\frac{ab}{\lambda_{0}^{2}}+\left( x_{0}-\frac{ab}{\lambda_{0}^{2}}\right)\frac{\sin \left( \lambda_{0}(1-s)\right) }{\sin \lambda_{0}} +\left( x-\frac{ab}{\lambda_{0}^{2}}\right)\frac{\sin \left( \lambda_{0}s\right) }{\sin \lambda_{0}}, \hspace{15pt} 0 \leq s \leq 1.\label{g3.1}
\end{equation}

\bigskip
$\bullet$ \hspace{5pt} $\sin\lambda_{0}= 0$, $\left\lbrace c_{i} \right\rbrace_{i=1}^{2} $ have no solution or uncountably many solutions.

\bigskip
Noting $\lambda_{0}=k\pi$ and the coefficient matrices of  equations (\ref{singular Coefficient})
\begin{equation*}
\begin{bmatrix}
1 & 0 & \vdots & x_{0}-\frac{ab}{\lambda_{0}^{2}}\\
 & & \vdots & \\
\cos\lambda_{0}+\frac{a}{\lambda_{0}}\sin\lambda_{0} & -\frac{2\theta \sin\lambda_{0}}{\lambda_{0}} & \vdots & x-\frac{ab}{\lambda_{0}^{2}}
\end{bmatrix}
\, = \,
\begin{bmatrix}
1 & 0 & \vdots & x_{0}-\frac{ab}{\lambda_{0}^{2}}\\
 & & \vdots & \\
(-1)^{k} & 0 & \vdots & x-\frac{ab}{\lambda_{0}^{2}}
\end{bmatrix},
\end{equation*}
one concludes that 
\begin{align*}
&\qquad  \left\lbrace c_{i} \right\rbrace_{i=1}^{2} \mbox{are solvable}\\
&\Longleftrightarrow x-\frac{ab}{\lambda_{0}^{2}}=(-1)^{k}\left(x_{0}-\frac{ab}{\lambda_{0}^{2}} \right) \\
&\Longleftrightarrow x-\frac{ab}{k^{2}\pi^{2}}=(-1)^{k}\left(x_{0}-\frac{ab}{k^{2}\pi^{2}} \right) \\
&\Longleftrightarrow
\left\{\begin{aligned}
& x=x_{0}, \hspace{15pt} \mbox{for $k$ even}\\
& x=\frac{2ab}{k^{2}\pi^{2}}-x_{0}, \hspace{15pt} \mbox{for $k$ odd}
\end{aligned}\right.,
\end{align*}
and the solutions are 
\begin{align*}
\left\{\begin{aligned}
& c_{1}=x_{0}-\frac{ab}{k^{2}\pi^{2}}\\
& c_{2}\in \R
\end{aligned}\right..
\end{align*}
By (\ref{x3}), the geodesics are
\begin{equation}
x(s)=\frac{ab}{\lambda_{0}^{2}}+\left(x_{0}-\frac{ab}{k^{2}\pi^{2}} \right)\left(\cos(\lambda_{0}s)+\frac{a}{\lambda_{0}}\sin(\lambda_{0}s)\right)-c_{2}\sin\left( \lambda_{0}s\right), \label{g3.2}
\end{equation}
where $c_{2}$ is an arbitrary real scalar, and the parameter $0 \leq s \leq 1$.

\bigskip
\begin{rmk}\label{rmk3}
Condition $\sin\lambda_{0}= 0$ means that $s=1$ is a singular regime, which corresponds to a hyperplane in the 2-d space-time coordinate $(x,s)$. A detailed discussion for more singular cases has been made in \cite{CF11} and \cite{F12}. We do not pursue this point here.
\end{rmk}

\bigskip
Finally, we summarise the explicit characterization of the geodesics as the following 
\begin{thm}\label{complete geodesics}
Suppose that $\theta >0$, $a, b\in \R$, $\rho \in \R \setminus \{0\}$. The geodesics of the perturbed Ornstein-Uhlenbeck operators $L=-\theta \partial_{x}^{2}+(ax+b)\partial_{x}+\rho x^{2}$ with the form $\left\lbrace x(s) | 0 \leq s \leq 1, \: x(0)=x_{0}, \: x(1)=x\right\rbrace $ are given by

\begin{enumerate}
\item $a^{2}+4\rho\theta > 0$.  $\left( \lambda_{0}:=\sqrt{a^{2}+4\rho\theta}\right)$ 
\begin{equation}
x(s) =-\frac{ab}{\lambda_{0}^{2}}+\left( x_{0}+\frac{ab}{\lambda_{0}^{2}}\right)\frac{\sinh \left( \lambda_{0}(1-s)\right) }{\sinh \lambda_{0}} +\left( x+\frac{ab}{\lambda_{0}^{2}}\right)\frac{\sinh \left( \lambda_{0}s\right) }{\sinh \lambda_{0}}. \tag{\ref{g1}}
\end{equation}

\bigskip
\item $a^{2}+4\rho\theta = 0$.
\begin{equation}
x(s)=\frac{ab}{2}s^{2}+\left( x-x_{0}-\frac{ab}{2}\right) s+x_{0}. \tag{\ref{g2}}
\end{equation}

\bigskip
\item $a^{2}+4\rho\theta < 0$.  $\left( \lambda_{0}:=\sqrt{-a^{2}-4\rho\theta}\right)$ 
      \begin{itemize}
      \item $\lambda_{0} \neq k\pi$ ($k \in \Z^{+}$)
      \begin{equation}
      x(s) =\frac{ab}{\lambda_{0}^{2}}+\left( x_{0}-\frac{ab}{\lambda_{0}^{2}}\right)\frac{\sin \left( \lambda_{0}(1-s)\right) }{\sin \lambda_{0}} +\left( x-\frac{ab}{\lambda_{0}^{2}}\right)\frac{\sin \left( \lambda_{0}s\right) }{\sin \lambda_{0}}. \tag{\ref{g3.1}}
      \end{equation}
      
      \bigskip
      \item $\lambda_{0} = k\pi$ ($k \in \Z^{+}$) and $x=\frac{ab}{k^{2}\pi^{2}}+(-1)^{k}\left(x_{0}-\frac{ab}{k^{2}\pi^{2}}\right)$
      \begin{equation}
      x(s)=\frac{ab}{\lambda_{0}^{2}}+\left(x_{0}-\frac{ab}{k^{2}\pi^{2}} \right)\left(\cos(\lambda_{0}s)+\frac{a}{\lambda_{0}}\sin(\lambda_{0}s)\right)-c_{2}\sin\left( \lambda_{0}s\right) \tag{\ref{g3.2}}
      \end{equation}
      with $c_{2}$ an arbitrary real scalar.
      \end{itemize}
\end{enumerate}
\end{thm}

\bigskip
In particular, we obtain the geodesics of $L^{\pm}$ in (\ref{L+}) and (\ref{L-}).

\begin{cor}
There exist unique geodesics for the perturbed Ornstein-Uhlenbeck operator $L^{+}=-\partial_{x}^{2}+x\partial_{x}+x^{2}$ fulfilling the boundary conditions $x(0)=x_{0}, \, x(1)=x$ with
\begin{equation}
x(s)=\frac{\sinh(\sqrt{5}(1-s))}{\sinh\sqrt{5}}x_{0}+\frac{\sinh(\sqrt{5}s)}{\sinh\sqrt{5}}x, \hspace{15pt} 0 \leq s \leq 1.
\end{equation}

\bigskip
\noindent
There exist unique geodesics for the perturbed Ornstein-Uhlenbeck operator $L^{-}=-\partial_{x}^{2}+x\partial_{x}-x^{2}$ fulfilling the boundary conditions $x(0)=x_{0}, \, x(1)=x$ with
\begin{equation}
x(s)=\frac{\sin(\sqrt{3}(1-s))}{\sin\sqrt{3}}x_{0}+\frac{\sin(\sqrt{3}s)}{\sin\sqrt{3}}x, \hspace{15pt} 0 \leq s \leq 1.
\end{equation}
\end{cor}

\bigskip
\section{Heat kernels}\label{kernel}

\bigskip
\subsection{A probabilistic ansatz}
We adopt a probabilistic ansatz to handle the heat kernel. It was first used in \cite{B99} for Kolmogorov operators, and recently for operators of Hermite type \cite{F12}. The perturbed Ornstein-Uhlenbeck operators
\begin{equation*}
L=-\theta \partial_{x}^{2}+(ax+b)\partial_{x}+\rho x^{2}
\end{equation*}
have heat kernels of the form
\begin{equation}
P(t; x_{0}, x)=\varphi(t)\exp\left\lbrace \alpha(t)x^{2}+\beta(t)xx_{0}+\gamma(t)x_{0}^{2}+\mu(t)x+\nu(t)x_{0}\right\rbrace \label{ansatz}
\end{equation}
where the coefficients $\varphi$, $\alpha$, $\beta$, $\gamma$, $\mu$, $\nu$ are functions of the time variable $t$ and satisfy the following systems of ordinary differential equations  
\begin{align}
&\dot{\alpha} = 4\theta \alpha^{2}-2a\alpha-\rho \label{eq1}\\
&\dot{\beta} = 4\theta \alpha \beta-a\beta \label{eq2}\\
&\dot{\gamma} = \theta \beta^{2} \label{eq3}\\
&\dot{\mu} = 4\theta \alpha \mu-a\mu-2b\alpha \label{eq4}\\
&\dot{\nu} = 2\theta \beta \mu-b\beta \label{eq5}\\
&\varphi^{-1}\dot{\varphi} = \theta \mu^{2}+ 2\theta \alpha-b\mu. \label{eq6}
\end{align}

\bigskip
\subsection{Identification of coefficients}
We now solve the equations (\ref{eq1})-(\ref{eq6}) and identify the free constant in each equation by comparing the coefficient with that in Hermitian case \cite{F12}, i.e. $a=0$. Without confusion, we denote the free constant by $C$ which may be different from time to time.

\bigskip
\subsubsection*{Case 1 \hspace{15pt}$a^{2}+4\rho\theta>0$.}

\bigskip
Equation (\ref{eq1}): Putting $\alpha_{1}=\frac{a+\sqrt{a^{2}+4\rho\theta}}{4\theta}$ and $\alpha_{2}=\frac{a-\sqrt{a^{2}+4\rho\theta}}{4\theta}$, one has 
\begin{align*}
&\qquad \dot{\alpha}=4\theta(\alpha-\alpha_{1})(\alpha-\alpha_{2})\\
\\
&\Longrightarrow \ln \left|\frac{\alpha-\alpha_{1}}{\alpha-\alpha_{2}}\right|=4\theta(\alpha_{1}-\alpha_{2})+C\\
\\
&\Longrightarrow  \alpha=\frac{a}{4\theta}+\frac{\sqrt{a^{2}+4\rho\theta}}{4\theta}\frac{1+Ce^{2\sqrt{a^{2}+4\rho\theta}}}{1-Ce^{2\sqrt{a^{2}+4\rho\theta}}}.
\end{align*}
Let $C=1$, one gets
\begin{equation}
\alpha(t)=\frac{a}{4\theta}-\frac{\sqrt{a^{2}+4\rho\theta}}{4\theta}\coth\left(\sqrt{a^{2}+4\rho\theta}\,t \right).\label{heat1.1}
\end{equation}

\bigskip
Equation (\ref{eq2}):
\begin{align*}
&\qquad \beta^{-1}\dot{\beta}=-\sqrt{a^{2}+4\rho\theta}\coth\left(\sqrt{a^{2}+4\rho\theta}\,t \right)\\
\\
&\Longrightarrow \ln \left|\beta\right|=-\ln\left(\sinh\left(\sqrt{a^{2}+4\rho\theta}\,t \right)  \right) +C\\
\\
&\Longrightarrow  \beta=\frac{C}{\sinh\left(\sqrt{a^{2}+4\rho\theta}\,t \right)}.
\end{align*}
Let $C=\frac{\sqrt{a^{2}+4\rho\theta}}{2\theta}$, one gets
\begin{equation}
\beta(t)=\frac{\sqrt{a^{2}+4\rho\theta}}{2\theta \sinh\left(\sqrt{a^{2}+4\rho\theta}\,t \right)}.
\end{equation}

\bigskip
Equation (\ref{eq3}):
\begin{align*}
&\qquad \dot{\gamma}=\frac{\left(\sqrt{a^{2}+4\rho\theta}\right) ^{2} }{4\theta\sinh^{2}\left(\sqrt{a^{2}+4\rho\theta}\,t \right)}\\
\\
&\Longrightarrow \gamma =-\frac{\sqrt{a^{2}+4\rho\theta}}{4\theta}\coth\left(\sqrt{a^{2}+4\rho\theta}\,t \right) +C.
\end{align*}
Let $C=\frac{a}{4\theta}$, one gets
\begin{equation}
\gamma(t)=\frac{a}{4\theta}-\frac{\sqrt{a^{2}+4\rho\theta}}{4\theta}\coth\left(\sqrt{a^{2}+4\rho\theta}\,t \right).
\end{equation}

\bigskip
Equation (\ref{eq4}): $\mu$ has the form
\begin{equation*}
\mu(t)=\frac{C(t)}{\sinh\left(\sqrt{a^{2}+4\rho\theta}\,t \right)},
\end{equation*}
where $C(t)$ satisfies
\begin{equation*}
\frac{\dot{C}}{\sinh\left(\sqrt{a^{2}+4\rho\theta}\,t \right)}=-2b\alpha(t).
\end{equation*}
By (\ref{heat1.1}), an integration yields
\begin{equation*}
C(t) =-2b\left(\frac{a\cosh\left(\sqrt{a^{2}+4\rho\theta}\,t \right)}{4\theta\sqrt{a^{2}+4\rho\theta}}-\frac{\sinh\left(\sqrt{a^{2}+4\rho\theta}\,t \right)}{4\theta}\right).
\end{equation*}
Hence,
\begin{equation}
\mu(t)=-\frac{ab}{2\theta\sqrt{a^{2}+4\rho\theta}}\coth\left(\sqrt{a^{2}+4\rho\theta}\,t \right)+\frac{b}{2\theta}.\label{heat1.4}
\end{equation}

\bigskip
Equation (\ref{eq5}):
\begin{align*}
&\qquad \dot{\nu}=-\frac{ab}{2\theta}\frac{\cosh\left(\sqrt{a^{2}+4\rho\theta}\,t \right)}{\sinh^{2}\left(\sqrt{a^{2}+4\rho\theta}\,t \right)}\\
\\
&\Longrightarrow \nu =\frac{ab}{2\theta\sqrt{a^{2}+4\rho\theta}\sinh\left(\sqrt{a^{2}+4\rho\theta}\,t \right)} +C.
\end{align*}
Let $C=\frac{b}{2\theta}$, one gets
\begin{equation}
\nu(t)=\frac{ab}{2\theta\sqrt{a^{2}+4\rho\theta}\sinh\left(\sqrt{a^{2}+4\rho\theta}\,t \right)}+\frac{b}{2\theta}.
\end{equation}

\bigskip
Equation (\ref{eq6}): By (\ref{heat1.1}) and (\ref{heat1.4}), an arrangement of right hand side of (\ref{eq6}) gets
\begin{align*}
\varphi^{-1}\dot{\varphi}&=\frac{a^{2}b^{2}}{4\theta(a^{2}+4\rho\theta)\sinh^{2}\left(\sqrt{a^{2}+4\rho\theta}\,t \right)}-\frac{\sqrt{a^{2}+4\rho\theta}}{2}\coth\left(\sqrt{a^{2}+4\rho\theta}\,t \right)\\
&\quad +\left(\frac{a}{2}-\frac{\rho b^{2}}{a^{2}+4\rho\theta} \right).
\end{align*}
Hence,
\begin{align*}
\varphi &=\frac{C}{\sinh^{1/2}\left(\sqrt{a^{2}+4\rho\theta}\,t \right)}\\
&\quad \cdot\exp\left\lbrace \left(\frac{a}{2}-\frac{\rho b^{2}}{a^{2}+4\rho\theta} \right)t-\frac{a^{2}b^{2}}{4\theta(a^{2}+4\rho\theta)^{3/2}}\coth\left(\sqrt{a^{2}+4\rho\theta}\,t \right)\right\rbrace.
\end{align*}
Let $C=\left(\frac{\sqrt{a^{2}+4\rho\theta}}{4 \pi \theta} \right) ^{1/2}$, one gets
\begin{align}
\notag
\varphi(t) &=\left(\frac{\sqrt{a^{2}+4\rho\theta}}{4\pi\theta\sinh\left(\sqrt{a^{2}+4\rho\theta}\,t \right)} \right) ^{1/2}\\
&\quad \cdot\exp\left\lbrace \left(\frac{a}{2}-\frac{\rho b^{2}}{a^{2}+4\rho\theta} \right)t-\frac{a^{2}b^{2}}{4\theta(a^{2}+4\rho\theta)^{3/2}}\coth\left(\sqrt{a^{2}+4\rho\theta}\,t \right)\right\rbrace.
\end{align}

\bigskip
\subsubsection*{Case 2 \hspace{15pt}$a^{2}+4\rho\theta=0$.}

\bigskip
Equation (\ref{eq1}): 
\begin{align*}
&\qquad \dot{\alpha}=4\theta\left( \alpha-\frac{a}{4\theta}\right)^{2}\\
\\
&\Longrightarrow \frac{1}{ \alpha-\frac{a}{4\theta}}=-4\theta \,t+C.
\end{align*}
Let $C=0$, one gets
\begin{equation}
\alpha(t)=\frac{a}{4\theta}-\frac{1}{4\theta \,t}.\label{heat2.1}
\end{equation}

\bigskip
Equation (\ref{eq2}):
\begin{align*}
&\qquad \beta^{-1}\dot{\beta}=-\frac{1}{t}\\
\\
&\Longrightarrow \beta=\frac{C}{t}.
\end{align*}
Let $C=\frac{1}{2\theta}$, one gets
\begin{equation}
\beta(t)=\frac{1}{2\theta \,t}.
\end{equation}

\bigskip
Equation (\ref{eq3}):
\begin{align*}
&\qquad \dot{\gamma}=\frac{1}{4\theta \,t^{2}}\\
\\
&\Longrightarrow \gamma=-\frac{1}{4\theta \,t}+C.
\end{align*}
Let $C=\frac{a}{4\theta}$, one gets
\begin{equation}
\gamma(t)=\frac{a}{4\theta}-\frac{1}{4\theta \,t}.
\end{equation}

\bigskip
Equation (\ref{eq4}): $\mu$ has the form
\begin{equation*}
\mu(t)=\frac{C(t)}{t},
\end{equation*}
where $C(t)$ satisfies
\begin{equation*}
\frac{\dot{C}}{t}=-2b\alpha(t).
\end{equation*}
By (\ref{heat2.1}), an integration yields
\begin{equation*}
C(t) =-\frac{ab}{4\theta}t^{2}+\frac{b}{2\theta}t.
\end{equation*}
Hence,
\begin{equation}
\mu(t)=-\frac{ab}{4\theta}t+\frac{b}{2\theta}.\label{heat2.4}
\end{equation}

\bigskip
Equation (\ref{eq5}):
\begin{align*}
&\qquad \dot{\nu}=-\frac{ab}{4\theta}\\
\\
&\Longrightarrow \nu =-\frac{ab}{4\theta}t +C.
\end{align*}
Let $C=\frac{b}{2\theta}$, one gets
\begin{equation}
\nu(t)=-\frac{ab}{4\theta}t+ \frac{b}{2\theta}.
\end{equation}

\bigskip
Equation (\ref{eq6}): 
\begin{align*}
&\qquad \varphi^{-1}\dot{\varphi}=\frac{a^{2}b^{2}}{16\theta}t^{2}-\frac{1}{2t}+\left(\frac{a}{2}-\frac{b^{2}}{4\theta} \right)\\
\\
&\Longrightarrow \varphi=\frac{C}{\sqrt{t}}\exp\left\lbrace \left(\frac{a}{2}-\frac{b^{2}}{4\theta} \right)t+\frac{a^{2}b^{2}}{48\theta}t^{3}\right\rbrace.
\end{align*}
Let $C=\left( \frac{1}{4\pi \theta}\right) ^{1/2}$, one gets
\begin{equation}
\varphi(t)=\left( \frac{1}{4\pi \theta \,t}\right) ^{1/2}\exp\left\lbrace \left(\frac{a}{2}-\frac{b^{2}}{4\theta} \right)t+\frac{a^{2}b^{2}}{48\theta}t^{3}\right\rbrace.
\end{equation}

\bigskip
\subsubsection*{Case 3 \hspace{15pt}$a^{2}+4\rho\theta<0$.}
Put $\alpha_{0}=\sqrt{-a^{2}-4\rho\theta}$.

\bigskip
Equation (\ref{eq1}): 
\begin{align*}
&\qquad \dot{\alpha}=4\theta\left( \alpha-\frac{a}{4\theta}\right)^{2}+\frac{\alpha_{0}^{2}}{4\theta}\\
\\
&\Longrightarrow \alpha=\frac{a}{4\theta}+\frac{\alpha_{0}}{4\theta}\tan(\alpha_{0}\,t+C).
\end{align*}
Let $C=\frac{\pi}{2}$, one gets
\begin{equation}
\alpha(t)=\frac{a}{4\theta}-\frac{\sqrt{-a^{2}-4\rho\theta}}{4\theta}\cot\left( \sqrt{-a^{2}-4\rho\theta} \,t\right).\label{heat3.1}
\end{equation}

\bigskip
Equation (\ref{eq2}):
\begin{align*}
&\qquad \beta^{-1}\dot{\beta}=-\alpha_{0}\cot(\alpha_{0}\,t)\\
\\
&\Longrightarrow \beta=\frac{C}{\sin(\alpha_{0}\,t)}.
\end{align*}
Let $C=\frac{\alpha_{0}}{2\theta}$, one gets
\begin{equation}
\beta(t)=\frac{\sqrt{-a^{2}-4\rho\theta}}{2\theta \sin\left( \sqrt{-a^{2}-4\rho\theta} \,t\right)}.
\end{equation}

\bigskip
Equation (\ref{eq3}):
\begin{align*}
&\qquad \dot{\gamma}=\frac{\alpha_{0}^{2}}{4\theta}\csc^{2}(\alpha_{0}\,t)\\
\\
&\Longrightarrow \gamma=-\frac{\alpha_{0}}{4\theta}\cot(\alpha_{0}\,t)+C.
\end{align*}
Let $C=\frac{a}{4\theta}$, one gets
\begin{equation}
\gamma(t)=\frac{a}{4\theta}-\frac{\sqrt{-a^{2}-4\rho\theta}}{4\theta}\cot\left( \sqrt{-a^{2}-4\rho\theta} \,t\right).
\end{equation}

\bigskip
Equation (\ref{eq4}): $\mu$ has the form
\begin{equation*}
\mu(t)=\frac{C(t)}{\sin(\alpha_{0}\,t)},
\end{equation*}
where $C(t)$ satisfies
\begin{equation*}
\frac{\dot{C}}{\sin(\alpha_{0}\,t)}=-2b\alpha(t).
\end{equation*}
By (\ref{heat3.1}), an integration yields
\begin{equation*}
C(t) =\frac{ab}{2\theta\alpha_{0}}\cos(\alpha_{0}\,t)+\frac{b}{2\theta}\sin(\alpha_{0}\,t).
\end{equation*}
Hence,
\begin{equation}
\mu(t)=\frac{ab}{2\theta\sqrt{-a^{2}-4\rho\theta}}\cot\left( \sqrt{-a^{2}-4\rho\theta} \,t\right)+\frac{b}{2\theta}.
\end{equation}

\bigskip
Equation (\ref{eq5}):
\begin{align*}
&\qquad \dot{\nu}=\frac{ab\cos(\alpha_{0}\,t)}{2\theta\sin^{2}(\alpha_{0}\,t)}\\
\\
&\Longrightarrow \nu =-\frac{ab}{2\theta\alpha_{0}\sin(\alpha_{0}\,t)} +C.
\end{align*}
Let $C=\frac{b}{2\theta}$, one gets
\begin{equation}
\nu(t)=-\frac{ab}{2\theta\sqrt{-a^{2}-4\rho\theta}\sin\left( \sqrt{-a^{2}-4\rho\theta} \,t\right)}+ \frac{b}{2\theta}.
\end{equation}

\bigskip
Equation (\ref{eq6}): 
\begin{align*}
&\qquad \varphi^{-1}\dot{\varphi}=\frac{a^{2}b^{2}}{4\theta\alpha_{0}^{2}}\csc^{2}(\alpha_{0}\,t)-\frac{\alpha_{0}}{2}\cot(\alpha_{0}\,t)+\left(\frac{a}{2}+\frac{\rho b^{2}}{\alpha_{0}^{2}} \right)\\
\\
&\Longrightarrow \varphi=\frac{C}{\sin^{1/2}(\alpha_{0}\,t)}\exp\left\lbrace \left(\frac{a}{2}+\frac{\rho b^{2}}{\alpha_{0}^{2}} \right)t-\frac{a^{2}b^{2}}{4\theta \alpha_{0}^{3}}\cot(\alpha_{0}\,t)\right\rbrace.
\end{align*}
Let $C=\left( \frac{\alpha_{0}}{4\pi \theta}\right) ^{1/2}$, one gets
\begin{align}
\notag
\varphi(t)&=\left( \frac{\sqrt{-a^{2}-4\rho\theta}}{4\pi \theta \sin\left( \sqrt{-a^{2}-4\rho\theta} \,t\right)}\right) ^{1/2}\\
&\quad \cdot \exp\left\lbrace \left(\frac{a}{2}-\frac{\rho b^{2}}{a^{2}+4\rho\theta} \right)t-\frac{a^{2}b^{2}}{4\theta (-a^{2}-4\rho\theta)^{3/2}}\cot\left( \sqrt{-a^{2}-4\rho\theta} \,t\right)\right\rbrace.
\end{align}

\bigskip
\subsection{Closed heat kernels}

Making use of ansatz (\ref{ansatz}), we summarise the results on coefficients in the following theorem. 
\begin{thm}\label{complete kernel}
Suppose that $\theta >0$, $a, b\in \R$, $\rho \in \R \setminus \{0\}$. The heat kernels $P(t; x, x_{0})$ of the perturbed Ornstein-Uhlenbeck operators $L=-\theta \partial_{x}^{2}+(ax+b)\partial_{x}+\rho x^{2}$ are given by

\begin{enumerate}
\item $a^{2}+4\rho\theta > 0$.
\begin{align}
\notag
&\quad P(t; x, x_{0})\\
\notag
&=\left(\frac{\sqrt{a^{2}+4\rho\theta}}{4\pi\theta\sinh\left(\sqrt{a^{2}+4\rho\theta}\,t \right)} \right) ^{1/2}\\
\notag
&\quad \cdot\exp\left\lbrace \left(\frac{a}{2}-\frac{\rho b^{2}}{a^{2}+4\rho\theta} \right)t-\frac{a^{2}b^{2}}{4\theta(a^{2}+4\rho\theta)^{3/2}}\coth\left(\sqrt{a^{2}+4\rho\theta}\,t \right)\right\rbrace\\
\notag
&\quad \cdot\exp\left\lbrace\left( \frac{a}{4\theta}-\frac{\sqrt{a^{2}+4\rho\theta}}{4\theta}\coth\left(\sqrt{a^{2}+4\rho\theta}\,t \right) \right) \left( x^{2}+x_{0}^{2}\right) \right\rbrace\\
\notag
&\quad \cdot\exp \left\lbrace \frac{\sqrt{a^{2}+4\rho\theta}}{2\theta \sinh\left(\sqrt{a^{2}+4\rho\theta}\,t \right)}xx_{0} \right\rbrace\\
\notag
&\quad \cdot\exp\left\lbrace\left( -\frac{ab}{2\theta\sqrt{a^{2}+4\rho\theta}}\coth\left(\sqrt{a^{2}+4\rho\theta}\,t \right)+\frac{b}{2\theta} \right)x \right\rbrace \\
&\quad \cdot\exp\left\lbrace\left( \frac{ab}{2\theta\sqrt{a^{2}+4\rho\theta}\sinh\left(\sqrt{a^{2}+4\rho\theta}\,t \right)}+\frac{b}{2\theta}\right)x_{0} \right\rbrace.\label{type1}
\end{align}

\bigskip
\item $a^{2}+4\rho\theta = 0$.
\begin{align}
\notag
&\quad P(t; x, x_{0})\\
\notag
&=\left( \frac{1}{4\pi \theta \,t}\right) ^{1/2}\exp\left\lbrace \left(\frac{a}{2}-\frac{b^{2}}{4\theta} \right)t+\frac{a^{2}b^{2}}{48\theta}t^{3}\right\rbrace\\
&\quad \cdot\exp\left\lbrace\left(  \frac{a}{4\theta}-\frac{1}{4\theta \,t}\right) \left( x^{2}+x_{0}^{2}\right)+ \frac{1}{2\theta \,t}xx_{0}+ \left( -\frac{ab}{4\theta}t+\frac{b}{2\theta}\right)\left( x+x_{0}\right)\right\rbrace.\label{type2}
\end{align}

\bigskip
\item $a^{2}+4\rho\theta < 0$.
\begin{align}
\notag
&\quad P(t; x, x_{0})\\
\notag
&=\left( \frac{\sqrt{-a^{2}-4\rho\theta}}{4\pi \theta \sin\left( \sqrt{-a^{2}-4\rho\theta} \,t\right)}\right) ^{1/2}\\
\notag
&\quad \cdot\exp\left\lbrace \left(\frac{a}{2}-\frac{\rho b^{2}}{a^{2}+4\rho\theta} \right)t-\frac{a^{2}b^{2}}{4\theta (-a^{2}-4\rho\theta)^{3/2}}\cot\left( \sqrt{-a^{2}-4\rho\theta} \,t\right)\right\rbrace\\
\notag
&\quad \cdot\exp\left\lbrace\left( \frac{a}{4\theta}-\frac{\sqrt{-a^{2}-4\rho\theta}}{4\theta}\cot\left( \sqrt{-a^{2}-4\rho\theta} \,t\right) \right) \left( x^{2}+x_{0}^{2}\right) \right\rbrace\\
\notag
&\quad \cdot\exp \left\lbrace \frac{\sqrt{-a^{2}-4\rho\theta}}{2\theta \sin\left( \sqrt{-a^{2}-4\rho\theta} \,t\right)}xx_{0} \right\rbrace\\
\notag
&\quad \cdot\exp\left\lbrace\left( \frac{ab}{2\theta\sqrt{-a^{2}-4\rho\theta}}\cot\left( \sqrt{-a^{2}-4\rho\theta} \,t\right)+\frac{b}{2\theta} \right)x \right\rbrace \\
&\quad \cdot\exp\left\lbrace\left( -\frac{ab}{2\theta\sqrt{-a^{2}-4\rho\theta}\sin\left( \sqrt{-a^{2}-4\rho\theta} \,t\right)}+ \frac{b}{2\theta}\right)x_{0} \right\rbrace.\label{type3}
\end{align}
\end{enumerate}
\end{thm}

\bigskip
\section{Further discussions on geodesics and heat kernels}\label{discussion}

\bigskip
\subsection{Critical cases}
The critical case, i.e. $a^{2}+4\rho\theta=0$ is a watershed in the discussion of singularities. Due to its independent interest, we formulate it as a proposition.
\begin{prop}
Suppose that $\theta >0$, $a, b\in \R$, $\rho \in \R \setminus \{0\}$ and $a^{2}+4\rho\theta=0$. Perturbed Ornstein-Uhlenbeck operators $L=-\theta \partial_{x}^{2}+(ax+b)\partial_{x}+\rho x^{2}$ have unique geodesics $\left\lbrace x(s) | 0 \leq s \leq 1\right\rbrace $ fulfilling boundary conditions $ x(0)=x_{0}, \: x(1)=x$ with
\begin{equation}
x(s)=\frac{ab}{2}s^{2}+\left( x-x_{0}-\frac{ab}{2}\right) s+x_{0},  \tag{\ref{g2}}
\end{equation}
and the heat kernels are given by 
\begin{align}
\notag
&\quad P(t; x, x_{0})\\
\notag
&=\left( \frac{1}{4\pi \theta \,t}\right) ^{1/2}\exp\left\lbrace \left(\frac{a}{2}-\frac{b^{2}}{4\theta} \right)t+\frac{a^{2}b^{2}}{48\theta}t^{3}\right\rbrace\\
&\quad \cdot\exp\left\lbrace\left(  \frac{a}{4\theta}-\frac{1}{4\theta \,t}\right) \left( x^{2}+x_{0}^{2}\right)+ \frac{1}{2\theta \,t}xx_{0}+ \left( -\frac{ab}{4\theta}t+\frac{b}{2\theta}\right)\left( x+x_{0}\right)\right\rbrace.\tag{\ref{type2}}
\end{align}
\end{prop}

\bigskip
\begin{rmk}\label{rmk2}
If $ab\neq 0$, there exists a unique parabola connecting $x_{0}$ and $x$. If $ab = 0$, we may assume $a\neq 0$ and $b=0$ since other cases are of Hermite type (cf. \cite{F12}). The operators in the family 
$$
\left\lbrace L=-\theta \partial_{x}^{2}+ ax\partial_{x}+\rho x^{2}| a^{2}+4\rho\theta=0 \right\rbrace 
$$ 
have exactly the same geodesic\textemdash a straight line connecting $x_{0}$ and $x$. In this case, one infers that $\rho=-\frac{a^{2}}{4\theta}<0$ is the minimum value that $\rho$ could take to keep the operators regular. In particular, perturbed Ornstein-Uhlenbeck operators $L=-\theta \partial_{x}^{2}+ ax\partial_{x}+\rho x^{2}$ share the same geodesics with Laplace operator $L_{0}=-\partial_{x}^{2}$, but the heat kernel of $L$ is extinct from Gaussian. In other words, geodesics do not suffice to characterize heat kernels in our cases. 
\end{rmk}

\bigskip
\subsection{Heat kernels in higher dimensions}
For brevity, we assort the heat kernels in Theorem \ref{complete kernel}.

\bigskip
\begin{definition}\label{def1}
For $L=-\theta \partial_{x}^{2}+(ax+b)\partial_{x}+\rho x^{2}$, formulae (\ref{type1}) is called  heat kernels of Type I and denoted by $P_{I}(t; x, x_{0})$. Similarly, (\ref{type2}) and (\ref{type3}) are called Type II and Type III, and denoted by $P_{II}(t; x, x_{0})$ and $P_{III}(t; x, x_{0})$ respectively.
\end{definition}

\bigskip
We next formulate heat kernels for operators in higher dimensions, and illustrate them to $L^{\pm}$ in higher dimensions.
\begin{prop}
Let 
\begin{equation*}
L=-\sum_{j=1}^{n}\theta_{j} \partial_{x_{j}}^{2}+\sum_{j=1}^{n}(a_{j}x_{j}+b_{j})\partial_{x_{j}}+\sum_{j=1}^{n}\rho_{j} x_{j}^{2}
\end{equation*}
be a perturbed Ornstein-Uhlenbeck operator in $n$ spatial variables, where $\theta_{j}>0$, $a_{j}, b_{j}\in \R $ and $\rho_{j}\in \R \setminus{0}$. The heat kernel $P(t; x, x_{0})$ is given by
\begin{equation*}
P(t; x, x_{0})=P_{I}(t; x, x_{0})P_{II}(t; x, x_{0})P_{III}(t; x, x_{0})
\end{equation*}
\end{prop}

\begin{proof}
Due to the independence of each spatial variable, the heat kernel $P$ is a product of $n$ 1-d heat kernels of the form (\ref{type1}), (\ref{type2}) and (\ref{type3}). According to the sign of the quantity $a^{2}+4\rho\theta$ and Definition \ref{def1}, these heat kernels can be assorted into 3 parts, i.e. $P_{I}$, $P_{II}$ and $P_{III}$, which completes the proof. 
\end{proof}

\bigskip
\begin{cor}
The $n$-dimensional perturbed Ornstein-Uhlenbeck operator $L^{+}=-\Delta+x\cdot\nabla+|x|^{2}$ has heat kernel
\begin{align*}
P(t; x, x_{0})&=\left(\frac{\sqrt{5}e^{t}}{4\pi\sinh(\sqrt{5}t)} \right) ^{\frac{n}{2}}\\
&\quad \cdot\exp\left\lbrace \frac{1-\sqrt{5}\coth(\sqrt{5}t)}{4}\left(|x|^{2}+|x_{0}|^{2} \right)+\frac{\sqrt{5}}{2\sinh(\sqrt{5}t)}x\cdot x_{0} \right\rbrace.
\end{align*}
The $n$-dimensional perturbed Ornstein-Uhlenbeck operator $L^{-}=-\Delta+x\cdot\nabla-|x|^{2}$ has heat kernel
\begin{align*}
P(t; x, x_{0})&=\left(\frac{\sqrt{3}e^{t}}{4\pi\sin(\sqrt{3}t)} \right) ^{\frac{n}{2}}\\
&\quad \cdot\exp\left\lbrace \frac{1-\sqrt{3}\cot(\sqrt{3}t)}{4}\left(|x|^{2}+|x_{0}|^{2} \right)+\frac{\sqrt{3}}{2\sin(\sqrt{3}t)}x\cdot x_{0} \right\rbrace. 
\end{align*}
\end{cor}

\bigskip
\section{Conclusions}
We first give direct answers to the questions proposed at the beginning of this paper. 
\begin{itemize}
\item The sign of the quantity $a^{2}+4\rho\theta$ decides whether singularities arise or not. Specifically, perturbed Ornstein-Uhlenbeck operators $L=-\theta \partial_{x}^{2}+ ax\partial_{x}+\rho x^{2}$ is regular if $a^{2}+4\rho\theta\geq 0$, and irregular if $a^{2}+4\rho\theta < 0$. Remark \ref{rmk1}, Remark \ref{rmk2} and Remark \ref{rmk3} make brief comparisons to Hermitian case on singularities. 
\item We characterize the singularities from the point of geodesics and heat kernels. Complete discussions for each case are summarised in Theorem \ref{complete geodesics} and Theorem \ref{complete kernel}.
\end{itemize} 

An obvious extension of this result would be to obtain geodesics and heat kernels for the operators like 
\begin{equation*}
L=-\theta \partial_{x}^{2}+(ax+b)\partial_{x}+\rho_{0} x^{2}+\rho_{1}x+\rho_{2}.
\end{equation*} 
For operators in a more general form 
\begin{equation*}
L=-\langle A\nabla, \nabla\rangle +\langle Bx,\nabla\rangle + \langle Cx,x \rangle ,
\end{equation*} 
this paper treats the diagonal case of the coefficient matrices. When $A$, $B$ and $C$ are not diagonal, some results have been reached. \cite{B99} essentially obtain the results for $A>0$ and $C>0$. However, $C<0$ is a specially interesting case to study, whose geometric analysis is absent in the literature.

\bigskip
\noindent
Sheng-Ya Feng\\
Department of Mathematics\\
East China University of Science and Technology\\
Shanghai 200237, P.R. China\\
\medskip
E-mail address: s.y.feng@ecust.edu.cn
\end{document}